\documentclass[12pt]{amsart}
\usepackage[a4paper,margin=1in]{geometry}
\usepackage{microtype}
\usepackage[cal=euler]{mathalfa}
\usepackage{tikz,tikz-cd}

\input{commands}
\usepackage[colorlinks,allcolors=blue]{hyperref}

\opr\mm{mod}

\def\bfg{{\bar\fg}} 
\def\bfn{\bar\fn}
\def\bfb{\bar\fb}
\def\bfh{\bar\fh}
\def\bta{\bar\ta}
\def\bnd{{\mathsf b}} 

\opr\II{\mathcal I} 
\opr\B{\mathcal B} 
\def\Ig{\II^\bZ} 
\def\IB{\II_\bnd} 
\def\IBg{\IB^\bZ} 
\def\IIx#1#2{\II_{#2}(#1)} 
\def\Modg{\Mod^\bZ} 
\def\modg{\mmod^\bZ} 
\def\modBg{\modg_\bnd} 

\opr\Bind{\bfB} 
\def\W{\bfW} 
\def\R{\bfR} 
\def\HW{\R}
\opr\intq{\mathbf{I}} 

\def\Qp{Q^\bQ_+} 
\def\Pp#1{P_+^{#1}} 
\def\d{^\vee} 

\def\A{\sA} 
\def\HH{\sH} 
\def\WW{\sW} 

\def\cong{\iso}
\def\gr#1{_{(#1)}} 

\begin{document}
\title{Global Weyl modules for thin Lie algebras are finite-dimensional}

\author{Vladimir Dotsenko}
\address{Institut de Recherche Math\'ematique Avanc\'ee, UMR 7501, Universit\'e de Strasbourg et CNRS, 7 rue Ren\'e-Descartes, 67000 Strasbourg, France}
\email{vdotsenko@unistra.fr}
\author{Sergey Mozgovoy}
\address{School of Mathematics, Trinity College Dublin, Dublin 2, Ireland
\newline\indent
Hamilton Mathematics Institute, Dublin 2, Ireland}
\email{mozgovoy@maths.tcd.ie}

\begin{abstract}
The notion of Weyl modules, both local and global, goes back to Chari and Pressley in the case of affine Lie algebras, and has been extensively studied for various Lie algebras graded by root systems. We extend that definition to a certain class of Lie algebras graded by weight lattices and prove that if such a Lie algebra satisfies a natural ``thinness'' condition, then already the global Weyl modules are finite-dimensional. Our motivating example of a thin Lie algebra is the Lie algebra of polynomial Hamiltonian vector fields on the plane vanishing at the origin. 
We also introduce stratifications of categories of modules over such Lie algebras and identify the corresponding strata categories.
\end{abstract}

\maketitle

\section{Introduction}

Weyl modules $W(\lambda)$  were originally introduced by Chari and Pressley \cite{MR1850556} for the extended affine Lie algebra $\bfg=\fg\otimes\bC[t,t^{-1}]\oplus\bC d$ associated to a finite-dimensional simple Lie algebra $\fg$; they are parametrized by dominant integral weights $\lambda$.
These modules were related to the Demazure modules of affine Kac--Moody algebras \cite{MR2271991}, and generalized to a plethora of contexts, an incomplete list of which includes twisted affine Lie algebras \cite{MR2423816,MR3055822}, algebras of $\fg$-valued currents on arbitrary affine algebraic variety~\cite{MR2102326}, Yangians~\cite{TanYangian,MR3334147}, toroidal Lie algebras~\cite{MR4160929,MR4681327}, currents valued in an arbitrary Kac--Moody Lie algebra~\cite{MR3520262}, Lie superalgebras~\cite{MR3951769,MR3981101}, Borel--de Siebenthal pairs \cite{MR3848450}, and Lie algebras of TKK type \cite{lau2023jordanalgebrasweightmodules}; generalizations for the case of non-dominant weights have also been considered \cite{MR3703467,MR3682816,MR3787562}.
More precisely, one can talk about global Weyl modules and local
Weyl modules. In both cases, one considers the category $\II(\bfg)$ of $\bfg$-modules that are integrable over $\fg$. For each dominant weight $\lambda$ of $\fg$, the global Weyl module $W(\la)\in\II(\bfg)$ is the maximal integrable module that possesses a cyclic vector $v$ of weight $\la$ and such that $xv=0$ for every $x\in \bfg$ of positive $\fh$-weight. Local Weyl modules are obtained from the global ones by imposing the condition that the cyclic vector $v$ is additionally an eigenvector for the weight zero subalgebra of $\bfg$.

All the examples above share a common feature: the algebra $\bfg$ whose representations one considers is, in a sense, made of several copies of the adjoint representation of
the Lie algebra $\fg$. Therefore~$\bfg$ is
graded by a root system~\cite{MR1161095}, and as such fits into the general framework of \cite{MR3765456}. In this paper, we consider a more general case of Lie algebras $\bfg$ that contain a finite-dimensional semisimple Lie algebra $\fg$ in a way that the adjoint action of $\fg$ on $\bfg$ is integrable, so that $\bfg$ is graded by the weight lattice of $\fg$. We shall however assume that all nonzero weights of the adjoint action $\fg$ on $\bfg$ are either positive or negative; for instance, this way one can consider various subalgebras of the Lie algebra of vector fields on the plane that preserves the origin, or the Lie algebras $\sl(\lambda)$, $\lambda\in\bC$, defined by Feigin~\cite{MR940679}.

Our original motivation came from thinking about the Lie algebra $\HH_2$ of polynomial Hamiltonian vector fields on the plane. Several different constructions of $\HH_2$-modules emerged in representation theory and algebraic geometry in the past years. For instance, that Lie algebra acts on global sections of various vector bundles on Hilbert schemes of points on the plane \cite{MR1918676}, on the ``$y$-ified Khovanov--Rozansky homology of links'' of~\cite{gorsky2024tautologicalclassessymmetrykhovanovrozansky}, and on vector spaces of the form $H^\bullet(X)[x,y]$, where $X$ is a moduli space of stable parabolic Higgs bundles of certain rank and degree for a generic stability condition~\cite{hausel2022pwh2}; in this latter case, this action was instrumental in one of the proofs of the $P=W$ conjecture.
It is probably reasonable to speculate that at least some of these actions emerge because the Lie algebra~$\HH_2$ is very close to the Lie algebra of the group of polynomial automorphisms of the
plane, see \cite{MR607583}.

The Lie algebra $\HH_2$ has no nontrivial finite-dimensional representations; however, it has an important Lie subalgebra $L_0(\HH_2)$ of Hamiltonian vector fields preserving the origin. That latter Lie algebra acts on finite-dimensional vector spaces of global sections of the above-mentioned vector bundles over
the punctual Hilbert schemes; moreover, the vector spaces $H^\bullet(X)[x,y]$ look like the underlying vector spaces of modules that are coinduced from that subalgebra (in reality the action is more complicated and is not coinduced). This raises a question of better understanding  finite-dimensional representations of the Lie algebra $L_0(\HH_2)$, and the question of studying Weyl modules and their quotients arises naturally. While studying these modules, we realized, that, by contrast with most previously studied cases (with the exception of some of the Weyl modules of \cite{MR3848450}), already all global Weyl modules over $L_0(\HH_2)$ are finite-dimensional.

In fact, once this was observed, we managed to prove that the global Weyl modules are finite-dimensional for a large class of algebras that contains $L_0(\HH_2)$. Namely, let us call the Lie algebra $\bfg$ \emph{thin} (with respect to $\fg$) if the centralizer of $\fg$ in $\bfg$ is zero and all irreducible $\fg$-modules have finite multiplicity in $\bfg$.
In the case of a thin algebra, a very general result on finite-dimensionality holds. To state that, we define the functor $\Bind_{\le\la}:\II(\fg)\to\II(\bfg)$ that sends an integrable $\fg$-module $V$ to the maximal quotient of the induced $\bfg$-module $\Ind^{\bfg}_{\fg}V$ that is $\fg$-integrable and $\la$-bounded. (In particular, the module $\Bind_{\le\la}(L(\la))$ coincides with
the global Weyl module $W(\la)$.) The general finite-dimensionality result that we prove is the following one.

\begin{itheorem}
If the adjoint action of $\fg$ on $\bfg$ is thin, then the functor $\Bind_{\le\la}$ maps finite-dimensional modules to finite-dimensional modules.
\end{itheorem}

Independently of the above result, we prove that the category $\IB(\bfg)\sbs\II(\bfg)$ of $\bfg$-modules that are integrable and bounded over $\fg$ can be equipped with a (left) stratification structure so that the global Weyl modules are the standard objects and the local Weyl modules are the proper standard objects of this stratification. (At this point we must warn the reader that the existing literature contains many closely related notions of  stratification of abelian categories
\cite{MR957457,cline_stratifying,khoroshkin2015highestweightcategoriesmacdonald,MR3338680,wiggins_stratified}.
Additionally, standardly stratified categories \cite{cline_stratifying,MR3338680} are often referred to as stratified.) The strata categories of this stratification are parametrized by dominant weights $\la\in P_+$
and can be identified with the categories $\Mod\A_\la$ of left modules over the associative algebras $\A_\la\iso\End_{\bfg}(W(\la))^\op$
originally introduced in \cite{MR2718936} for the Lie algebras of the form $\bfg=\fg\ts A$, where $A$ is commutative.
If $\bfg=\bop_{i\ge 0}\bfg\gr i$ is non-negatively graded,
with $\bfg\gr0=\fg$ and finite-dimensional components $\bfg\gr i$,
then the category $\modBg(\bfg)$ of $\bZ$-graded $\bfg$-modules (bounded over~$\fg$) with finite-dimensional components
can be also equipped with a left stratification structure.
We will show that this structure is a full stratification
if $\bfg$ has an automorphism that restricts to $-\id$ on $\fh$
(this is true, in particular, for $\bfg=L_0(\HH_2)$).
\medskip

The paper is organized as follows. In Section \ref{sec:int-bounded}, we fix the notational conventions concerning semisimple Lie algebras and prove an adjunction result involving maximal integrable quotients of $\fg$-modules. In Section \ref{sec:findim}, we introduce the functor $\Bind_{\le\la}$, prove the finite-dimensionality results, and give several examples of situations where this result applies.
In Section \ref{sec:stratifications}, we introduce left recollements
and left stratifications of abelian categories.
We introduce the algebras~$\A_\la$,
generalize to our context the categorical approach of \cite{MR2718936,MR3765456},
and equip the category $\IB(\bfg)$ with a left stratification structure with strata categories $\Mod\A_\la$.
Moreover, we show for $\bZ$-graded Lie algebras
$\bfg=\bop_{i\ge 0}\bfg_{(i)}$
that the category $\modBg(\bfg)$ has a full stratification structure
if $\bfg$ satisfies some additional conditions.
In Section \ref{sec:local}, we discuss appropriate local Weyl modules, particularly those arising from augmentations $\bfg\to\fg$. We conclude with a conjecture motivated by experimental data on local Weyl modules for the Lie algebra $L_0(\HH_2)$:
it appears that the socle of every local Weyl module for that Lie algebra is isomorphic to the defining two-dimensional $\sl_2$-module.


\subsection*{Acknowledgements. } A large part of this work was done during visits of the first author to Trinity College Dublin in February and April 2024 and the second author to the University of Strasbourg in March 2024. We thank these institutions for excellent working conditions. Research of the first author was supported by Institut Universitaire de France.

The intention of studying Weyl modules for the Lie algebra of vector fields on $\bC^2$ was indicated to the first author by Boris Feigin and Sergey Loktev some 20 years ago; as we understand from conversations with them, that project never led to any published work, and did not reach the stage where finite-dimensionality of global Weyl modules could be conjectured. We thank them for useful discussions. We also thank Anton Khoroshkin for useful discussions of \cite{feigin2023peterweyltheoremiwahorigroups}, for valuable comments on a draft version of the paper, and for drawing our attention to his unpublished work \cite{khoroshkin2015highestweightcategoriesmacdonald}.
The first author thanks Vera Serganova for interesting discussions of a draft version of the paper.
All the computations given in examples below were performed in \texttt{Magma} \cite{MR1484478}.

\section{Integrable and bounded modules}\label{sec:int-bounded}

\subsection{Notation}
Throughout the paper, we use the following notation.
We denote by $\fg$ a semisimple Lie algebra with a Cartan subalgebra $\fh\sbs\fg$, a root system $\De=\De_+\sqcup\De_-\sbs\fh^*$, coroots $\al\d\in\fh$ for $\al\in\De$, and Chevalley generators $(e_i,f_i,h_i)_{i=1}^r$.
We let $Q=\sum_{\al\in\De}\bZ\al$ and $P=\sets{\la\in\fh^*}{\ang{\la,\al\d}\in\bZ\ \forall \al\in\De}$ be, respectively, the root lattice and the weight lattice of $\fg$. Furthermore, let $Q_+=\sum_{\al\in\De_+}\bZ_{\ge0}\al$ and
$P_+=\sets{\la\in P}{\ang{\la,\al\d}\in\bZ_{\ge0}\ \forall \al\in\De_+}$
be, respectively, the positive root monoid and the dominant weight monoid.
The monoid $P_+$ is contained in the positive root cone
$\Qp=\sum\nolimits_{\al\in\De_+}\bQ_{\ge0}\al$.
We shall consider the partial order on $P$ defined by $\la\le\mu$ if $\mu-\la\in \Qp$.
This poset is locally finite since the intervals
\[P_{\la,\mu}=\sets{\nu\in P}{\la\le\nu\le\mu}=P\cap(\la+\Qp)\cap(\mu-\Qp)\]
are finite.
In particular, the set
$\Pp{\le\la}=\sets{\mu\in P_+}{\mu\le\la}\sbs P_{0,\la}$ is finite.
We may have $\Pp{\le\la}\ne\es$ for $\la\not\in P_+$; for instance, if we consider $\fg=\sl_3$, so that $Q=\bZ\al_1\oplus\bZ\al_2$
with $(\al_1,\al_1)=(\al_2,\al_2)=2$, $(\al_1,\al_2)=-1$ and $P=\bZ\om_1\oplus\bZ\om_2$,
where $\om_1=\frac13(2\al_1+\al_2)$ and
$\om_2=\frac13(\al_1+2\al_2)$, we have $\om_1\in \Pp{\le\la}$ for
$\la=3\om_1-\om_2$.

\subsection{Integrable modules}
A \fg-module $V$ is called \idef{weighted}
if $V=\bop_{\la\in\fh^*}V_\la$, where
$V_\la=\sets{v\in V}{hv=\la(h)v\ \forall h\in \fh}$ is the weight $\la$ subspace.
A weighted \fg-module $V$ is called \emph{integrable} \cite[\S3.6]{MR1104219} if all $e_i$ and $f_i$ are locally-nilpotent on $V$.
One can show \cite[\S3.8]{MR1104219} that integrability is equivalent to either of the following properties:
\begin{itemize}
\item Every vector of $V$ is contained in a finite-dimensional $\fg$-submodule.
\item $V$ is a direct sum of finite-dimensional $\fg$-modules.
\end{itemize}
Note that under these conditions we have $P(V)=\sets{\la\in\fh^*}{V_\la\ne0}\sbs P$.

For a weight $\la\in P$, a weighted $\fg$-module $V$ is called \emph{$\la$-bounded} if $P(V)\sbs\la-\Qp$. Furthermore, such module is said to be \emph{bounded} if $P(V)\sbs\bigcup_{i=1}^n(\la_i-\Qp)$ for finitely many weights $\la_i\in P$.
Note that an integrable $\fg$-module $V$ is $\la$-bounded if and only if $\mu\le\la$ for every irreducible summand $L(\mu)$ of $V$.


\begin{lemma}\label{lm:int=ib}
For a weighted $\fg$-module~$V$ the following conditions are equivalent
\begin{enumerate}
\item $P(V)$ is finite.
\item $V$ is integrable and bounded.
\end{enumerate}
\end{lemma}
\begin{proof}
The implication (1) $\Rightarrow$ (2) is trivial. Assuming (2), we consider the action of the Weyl group $W$ on $P(V)$.
Let $w_0\in W$ be the longest element so that $w_0(Q_+)=-Q_+$.
Then $P(V)\sbs\bigcup_{i,j}(w_0\la_i+\Qp)\cap (\la_j-\Qp)$
and the sets $P_{\la,\mu}=P\cap(\la+\Qp)\cap(\mu-\Qp)$ are finite.
\end{proof}

In what follows, we denote by $\II(\fg)$ the category of integrable \fg-modules, by $\B(\fg)$ the category of bounded $\fg$-modules, and by $\IB(\fg)$ the category of integrable and bounded $\fg$-modules.
The BGG category $\cO\sbs\B(\fg)$ consists of modules having finite-dimensional weight spaces.
The intersection $\cO\cap\cI(\fg)$ consists of finite-dimensional $\fg$-modules.
For $V\in\B(\fg)$, let $\intq(V)$ denote the maximal integrable quotient of $V$.
In view of Lemma \ref{lm:int=ib}, this module may be constructed as follows. Since $V$ is bounded, we have $P(V)\sbs\bigcup_{i}(\la_i-\Qp)$ for finitely many $\la_i\in P$,
and we may consider the finite set $\Om=\bigcup_{i,j}P_{w_0\la_i,\la_j}$
and define $\intq(V)$ as the quotient of $V$ by the \fg-submodule generated by $\bop_{\mu\notin \Om}V_\mu$.

\begin{lemma}
The maximal integrable quotient functor $\intq\colon\B(\fg)\to \IB(\fg)$ is left adjoint to the embedding functor $\imath\colon\IB(\fg)\to \B(\fg)$.
\end{lemma}
\begin{proof}
Let $V$ be a bounded module and $U$ be an integrable and bounded module. Consider a morphism $f\colon V\to U$. The module $V'=f(V)\sbs U$ is integrable and bounded, hence the map $f$ factors through the maximal integrable quotient of $V$. Thus, we have
 \[
\Hom_{\IB}(\intq(V),U)\cong \Hom_{\B}(V,\imath(U))=\Hom_{\B}(V,U),
 \]
as desired.
\end{proof}

For $\la\in P_+$, let $L(\la)\in\II(\fg)$ denote the irreducible $\fg$-module of highest weight $\la$.
Every $V\in\II(\fg)$ can be written in the form $V=\bop_{\la\in P_+}V^{(\la)}\ts L(\la)$, where $V^{(\la)}=\Hom_{\fg}(L(\la),V)$.
The number $[V\col L(\la)]=\dim V^{(\la)}$ is the \emph{multiplicity} of $L(\la)$ in $V$.


\section{Finite-dimensionality of global Weyl modules}\label{sec:findim}

We now fix the following set-up that will be assumed throughout the paper. Let $\bfg$ be a Lie algebra containing $\fg$ as a subalgebra and assume that the adjoint action of $\fg$ on $\bfg$ is integrable.
Let
\[
\bfh=\bfg_0,\qquad
\bfn_{-}=\bop_{\la<0}\bfg_\la,\qquad
\bfn_{+}=\bop_{\la>0}\bfg_\la,
\]
and assume that
$\bfg=\bfn_-\oplus\bfh\oplus\bfn_+$.

\begin{example}\label{H2}
As indicated in the introduction, our original motivating example was the Lie algebra $L_0(\HH_2)$ of polynomial Hamiltonian vector fields on the plane vanishing at the origin.
The algebra~$\HH_2$ is the Lie algebra
(with $\dd_x=\pd x$, $\dd_y=\pd y$, $f_x=\dd_xf$, $f_y=\dd_y f$)
\[\HH_2=\sets{X_f=f_x\dd_y-f_y\dd_x}{f\in\bC[x,y]}.\]
This Lie algebra has a grading given by $\deg X_f=\deg f-2$, and the algebra $L_0(\HH_2)\sbs\HH_2$ consists of vector fields of non-negative degree.
We can identify $L_0(\HH_2)$ with $\bC[x,y]_{\ge2}$ (spanned by homogeneous polynomials of degree $\ge2$) equipped with the Lie bracket
\[[f,g]=f_xg_y-f_yg_x.\]
The degree zero component of $\bfg=\bop_{i\ge0}\bfg\gr i=L_0(\HH_2)$ is isomorphic to $\sl_2$
\[\sl_2=\bC\set{e,f,h}\isoto \bfg\gr0,\qquad
e\mto x\dd_y,\quad f\mto y\dd_x,\quad h\mto x\dd_x-y\dd_y.\]
The adjoint action of $\sl_2$ on $\bfg$ is integrable.
In fact, we have an isomorphism of $\sl_2$-modules $\bfg\gr i\iso L(i+2)$ for all $i\ge0$.
The algebra $\bfg$ has a basis
$e_{ij}=x^{i+1}y^{j+1}$, for $i,j\ge1$ and $i+j\ge0$,
with the bracket $[e_{ij},e_{kl}]
=\begin{vmatrix}i+1&j+1\\k+1&l+1\end{vmatrix}e_{i+k,j+l}$.
In particular, $h=-e_{00}$ and $[e_{00},e_{ij}]=(j-i)e_{ij}$, hence $e_{ij}$ has weight $i-j$ and degree $i+j$.
Note that $\bfh=\bop_{i\ge0}\bC e_{ii}$ is an abelian subalgebra.
\end{example}

\begin{remark}\label{rm:sep}
In the case of a Lie algebra $\bfg$ integrable over $\fg=\sl_2$ (like $\bfg=L_0(\HH_2)$), the assumption $\bfg=\bfn_-\oplus\bfh\oplus\bfn_+$ is automatically satisfied since $P\iso\bZ$. This phenomenon does not persist for other Lie algebras of vector fields from the Cartan classification~\cite{MR360732}. In particular, it is easy to check that for each of the two natural generalizations of the Lie algebra $\HH_2$, namely the Lie algebra $\bfg=\HH_{2n}$ of polynomial Hamiltonian vector fields on the $2n$-dimensional affine space (with its subalgebra $\fg=\mathfrak{sp}_{2n}$) and the Lie algebra $\bfg=\mathsf{S}_{n+1}$ of polynomial vector fields of zero divergence on the $(n+1)$-dimensional affine space (with its subalgebra $\fg=\mathfrak{sl}_{n+1}$), the direct sum $\bfn_-\oplus\bfh\oplus\bfn_+$ is a proper subalgebra of $\bfg$ for $n>1$. It is tempting to find a way to separate weights into positive and negative, for instance to fix a vector $x\in\fh$ for which $\langle\alpha,x\rangle>0$ for all $\alpha\in\Delta_+$ and $\langle\lambda,x\rangle=0$ for $\lambda\in P$ implies $\lambda=0$, consider the decomposition
 \[
\bfg=\bigg(\bop_{\langle\lambda,x\rangle< 0}\bfg_\la\bigg)
\oplus \bfg_0\oplus \bigg(\bop_{\langle\lambda,x\rangle>0}\bfg_\la\bigg)
 \]
and then modify the partial order on the weight lattice using $x$-positive weights appearing in $\bfg$
(\cf \cite{khoroshkin2015highestweightcategoriesmacdonald}).
 However, already for the Lie algebra $\HH_4$ the corresponding partial order is not locally finite, so a generalization of our results to that case is far from straightforward.
\end{remark}



The following definition of the global Weyl module usually appears in the literature.

\begin{definition}\label{class GW}
For $\la\in P$, the \idef{global Weyl module} $W(\la)$ is the
maximal $\fg$-integrable $\bfg$-module generated by the vector $v_\la$ of weight $\la$
such that $\bfn_+v_\la=0$.
\end{definition}

For example, if $A$ is a unital commutative algebra, we may consider the Lie algebra $\bfg=\fg\ts A$, for which the global Weyl modules $W(\la)$  are the modules introduced in \cite[Def.~3]{MR2102326}. Note that if $W(\la)\ne0$, then $\la$ is dominant. Indeed, $W(\la)$ is integrable and $\la\in P(W(\la))$. We have $\la+\al_i\notin P(V)$, hence $\ang{\la,\al_i^\vee}\ge0$ by \cite[Cor.~3.6]{MR1104219}.

\subsection{Truncations}
We shall now define functors which give a categorical context for Weyl modules.
For $V\in\II(\fg)$, we define the \emph{$\la$-truncation} $\ta_{\le\la}(V)$ and the \emph{complementary $\la$-truncation} $\bta_{\le\la}(V)$ by the formulas
\[\ta_{\le\la}(V)=\bop_{\mu\le\la}V^{(\mu)}\ts L(\mu),\qquad
\bta_{\le\la}(V)=\bop_{\mu\not\le\la}V^{(\mu)}\ts L(\mu).\]

Let $\II(\bfg)$ denote the category of $\bfg$-modules that are integrable over $\fg$. The forgetful functor $\II(\bfg)\to\II(\fg)$ has the left adjoint \[\Ind_{\fg}^{\bfg}\colon\II(\fg)\to \II(\bfg),\qquad
V\mto U(\bfg)\ts_{U(\fg)}V.\]

Let $\IIx\bfg{\le \la}$ denote the full subcategory of $\II(\bfg)$ consisting of modules $V$ such that $V=\ta_{\le\la}V$.
The embedding functor $\IIx\bfg{\le\la}\emb \II(\bfg)$ has the left adjoint
 \[
\ta_{\le\la}^{\bfg} \colon \II(\bfg)\to\IIx\bfg{\le\la},\qquad V\mto V/U(\bfg)\bta_{\le\la} V.
 \]
The composition
 \[
\Bind_{\le\la}=\ta_{\le\la}^{\bfg}\Ind_{\fg}^{\bfg}\colon \II(\fg)\to\IIx\bfg{\le\la}
 \]
is left adjoint to the forgetful functor
$\IIx\bfg{\le\la}\to\II(\fg)$.

As we shall see below, all our proofs only use the fact that the set $\Pp{\le\la}$ is finite. In fact, the proofs become somewhat shorter and cleaner if we adopt that level of generality and define,
for any finite set $\Om\sbs P_+$,
\[\ta_\Om(V)=\bop_{\mu\in\Om}V^{(\mu)}\ts L(\mu),\qquad
\bta_\Om (V)=\bop_{\mu\not\in\Om}V^{(\mu)}\ts L(\mu).\]
We denote by $\IIx\bfg\Om\sbs\II(\bfg)$ the category of objects $V\in\II(\bfg)$
such that $V=\ta_\Om(V)$.
The embedding functor $\IIx\bfg{\Om}\emb \II(\bfg)$ has the left adjoint
\begin{equation}\label{trunc2}
\ta_\Om^{\bfg}\colon \II(\bfg)\to\IIx\bfg{\Om},\qquad V\mto V/U(\bfg)\bta_\Om V.
\end{equation}
We consider the composition
\begin{equation}\label{W_Om}
\Bind_\Om=\ta_{\Om}^{\bfg}\Ind_{\fg}^{\bfg}\colon
\II(\fg)\to\IIx\bfg\Om.
\end{equation}

Note that a left adjoint functor maps projective objects to projective objects.
Therefore all objects in the image of $\Bind_\Om$ (in particular, in the image of $\Bind_{\le\la}$) are projective.
The following result relates the classical definition of Weyl modules to the functor $\Bind_{\le\la}$.

\begin{proposition}\label{prop:usualWeyl}
There is an isomorphism of $\bfg$-modules
 \[
 W(\la)\cong \Bind_{\le\la}(L(\la)).
 \]
\end{proposition}

\begin{proof}
Consider the Lie algebras $\bfb=\fh\oplus\bfn_+$
and $\bfb_-=\bfh\oplus\bfn_-$.
Let $\bfb$ act on $\bC v_\la$ through $\bfb\to\fh\xto\la\bC$
and let $N(\la)=\Ind_{\bfb}^{\bfg} (\bC v_\la)$. We have $\bfb+\bfb_-=\bfg$, hence $N(\la)$ is generated by $v_\la$ under the action of $U(\bfb_-)$. Therefore $N(\la)$ is $\la$-bounded, and we have $W(\la)=\intq (N(\la))$ by Definition \ref{class GW}.
On the other hand, we note that the module $W(\la)$ is integrable and $\la$-bounded, hence its irreducible summands $L(\mu)$ satisfy $\mu\le\la$.
The highest weight vector $v_\la\in N(\la)$ induces a morphism $L(\la)\to N(\la)$ of $\fg$-modules, hence
a surjective morphism
 \[
\Bind_{\le\la}(L(\la))=\ta_{\le\la}^{\ub\fg}\Ind^{\ub\fg}_\fg L(\la)\to W(\la).
 \]
of $\bfg$-modules. The module $\Bind_{\le\la}(L(\la))$ is integrable, generated by $v_\la$ and is $\la$-bounded, hence $\bfn_+v_\la=0$.
By the maximality of $W(\la)$, we conclude that the surjective map $\Bind_{\le\la}(L(\la))\to W(\la)$ is an isomorphism.
\end{proof}

\begin{lemma}\label{lm:weyl prop}
For all $V\in\IIx\bfg{\le\la}$, we have
\[\Hom_{\bfg}(W(\la),V)\iso V_\la.\]
\end{lemma}
\begin{proof}
Indeed,
$\Hom_{\bfg}(W(\la),V)=\Hom_{\bfg}(\Bind_{\le\la}L(\la),V)
\iso\Hom_\fg(L(\la),V)\iso V_\la$.
\end{proof}

The following lemma gives an explicit construction of global Weyl modules. It is useful for doing explicit computations with those modules; for instance, it is this definition that we implemented in \texttt{Magma} to compute various examples throughout the paper. However, we shall see that for  theoretical purposes, one can entirely bypass this construction and only use the universal properties of the Weyl modules.

\begin{lemma}
The global Weyl module $W(\la)$ is isomorphic to the cyclic module with a cyclic vector $v_\la$ such that
 \[
\bfn_+v_\la=0, \qquad
h(v_\la)=\la(h)v_\la \ \forall h\in \fh, \qquad f_i^{\la(h_i)+1}v_\la=0 \ \forall i=1,\ldots,r.
 \]
\end{lemma}

\begin{proof}
This immediately follows from the standard fact that $L(\la)$ is the quotient of $U(\fg)$ by the left ideal generated by
 \[
\fn_+v_\la=0, \qquad
h(v_\la)=\la(h)v_\la \ \forall h\in \fh, \qquad f_i^{\la(h_i)+1}v_\la=0 \ \forall i=1,\ldots,r.
 \]
\end{proof}

\subsection{Global Weyl modules for \texorpdfstring{$L_0(\HH_2)$}{L0H2}}\label{sec:globalH2}
For the graded Lie algebra
$\bfg=\bop_{i\ge0}\bfg\gr i=L_0(\HH_2)$, we can interpret the Weyl modules $W(\la)$ as $\bZ$-graded modules over $\bfg$.
The following are $\sl_2$-decompositions of global Weyl modules for the Lie algebra $L_0(\HH_2)$ for $\lambda=0,\ldots,7$, listed degree by degree.

 \[
W(0)=
\begin{tabular}{c}
$L(0)$
\end{tabular},\quad
W(1)=\begin{tabular}{c}
$L(1)$
\end{tabular},\quad
W(2)=\begin{tabular}{c}
$L(2)$\\
$L(1)$
\end{tabular},\quad
W(3)=\begin{tabular}{c}
$L(3)$\\
$L(2)\oplus L(0)$\\
$L(3)\oplus L(1)$
\end{tabular},
 \]

 \[
W(4)=\begin{tabular}{c}
$L(4)$\\
$L(3)\oplus L(1)$\\
$L(4)\oplus 2L(2)\oplus L(0)$\\
$L(3)\oplus 2L(1)$
\end{tabular},\quad
W(5)=\begin{tabular}{c}
$L(5)$\\
$L(4)\oplus L(2)$\\
$L(5)\oplus 2L(3)\oplus 2L(1)$\\
$2L(4)\oplus 3L(2)\oplus 2L(0)$\\
$L(5)\oplus 2L(3)\oplus 3L(1)$
\end{tabular},
 \]

 \[
W(6)=\begin{tabular}{c}
$L(6)$\\
$L(5)\oplus L(3)$\\
$L(6)\oplus 2L(4)\oplus 2L(2)\oplus L(0)$\\
$2L(5)\oplus 4L(3)\oplus 3L(1)$\\
$2L(6)\oplus 3L(4)\oplus 6L(2)\oplus 3L(0)$\\
$2L(5)\oplus 4L(3)\oplus 4L(1)$
\end{tabular},\quad
W(7)=\begin{tabular}{c}
$L(7)$\\
$L(6)\oplus L(4)$\\
$L(7)\oplus 2L(5)\oplus 2L(3)\oplus L(1)$\\
$2L(6)\oplus 4L(4)\oplus 4L(2)\oplus L(0)$\\
$2L(7)\oplus 4L(5)\oplus 7L(3)\oplus 6L(1)$\\
$3L(6)\oplus 7L(4)\oplus 9L(2)\oplus 5L(0)$\\
$L(7)\oplus 4L(5)\oplus 7L(3)\oplus 7L(1)$
\end{tabular}
 \]

In particular, all these modules are finite-dimensional. As we shall see below, this is always the case for $L_0(\HH_2)$, as well as a large class of other Lie algebras.

\subsection{Main finite-dimensionality result}\label{sec:main}

The category $\II(\fg)$ is closed under tensor product, hence it makes sense to talk about algebras in this category.
For an associative algebra $A$ in $\cI(\fg)$, let $\Mod(A,\II(\fg))$ be the category of left $A$-modules in $\II(\fg)$ and
$\Mod(A,\IIx\fg\Om)$ be the category of left $A$-modules in $\IIx\fg\Om$, for a finite subset $\Om\sbs P_+$.
The embedding functor $\Mod(A,\IIx\fg\Om)\emb \Mod(A,\II(\fg))$ has the left adjoint
\[
\ta_\Om^A \colon \Mod(A,\II(\fg))\to \Mod(A,\IIx\fg\Om),\qquad
M\mto M/A\bta_\Om(M).
\]
For example, if $\bfg=\fg\oplus I$, where $I$ is an ideal, then
$A=U(I)$ is an associative algebra in the category $\II(\fg)$.
We have $\Mod(A,\II(\fg))\iso\II(\bfg)$
and $\ta_\Om^A$ coincides with the truncation $\ta_\Om^{\bfg}$ given by \eqref{trunc2}.


\begin{proposition}\label{prop:com-alg-fd}
If $V$ is a finite-dimensional $\fg$-module such that $[V\col L(0)]=0$,
then the algebra $S_{\Om}(V)=(SV)/(SV\cdot \bta_\Om(SV))$ is finite-dimensional.
\end{proposition}
\begin{proof}
The $\fg$-module $A=S_{\Om}(V)$ has only irreducible summands $L(\mu)$ with $\mu\in\Om$.
Therefore it has a weight decomposition $A=\bop_{\mu\in P}A_\mu$
such that $P(A)=\sets{\mu\in P}{A_\mu\ne0}$ is finite.
Consider the weight decomposition $V=\bop_{\mu\in P}V_\mu$
and let $A_{(\mu)}$ be the image of $S(V_\mu)$ in $A$.
If $\mu\ne0$, then $S^k(V_\mu)$ is mapped to $A_{k\mu}$ which is zero for $k\gg0$.
Therefore $A_{(\mu)}$ is finite-dimensional.
We have $A_{(\mu)}=\bC$ for all but a finite number of $\mu\in P$, hence $B=\prod_{\mu\ne0}A_{(\mu)}\sbs A$ is finite-dimensional.

The algebra $A$ is generated over the algebra $B$ by the image of $V_0$.
Let us show that the elements of this image are integral over $B$.
The assumption $[V:L(0)]=0$ implies that $V_0$ is spanned by elements
of the form $u=fv$, where $f\in\fg_{-\al}$ and $v\in V_\al$ for some $\al\in\De$.
The endomorphism $f$ of $V$ induces a derivation of $S(V)$.
This derivation preserves $\bta_\Om (SV)$ and the ideal $SV\cdot \bta_\Om(SV)$ (which are both $\fg$-modules).
Therefore $f$ induces a derivation of $A$.
Let $\bar u,\bar v\in A$ be the images of $u,v\in V$.
The element $f^m(\bar v^m)$ is a linear combination of products
$\prod_{i=1}^m f^{k_i}(\bar v)$, where $k_i\ge0$ and $\sum_i k_i=m$.
We have $f^k(\bar v)\in B$ if $k\ne1$,
hence
$\prod_{i=1}^m f^{k_i}(\bar v)\in B\bar u^r$,
where $r=\#\sets{i}{k_i=1}$.
For $m\gg0$ we have $\bar v^m=0$, hence
$0=f^m (\bar v^m)\in m!\cdot \bar u^m+\sum_{r=0}^{m-1}B\bar u^r$.
Therefore $\bar u^m\in \sum_{r=0}^{m-1}B\bar u^r$ and $\bar u$ is integral over $B$.
This implies that the algebra $A$ is generated over $B$ by finitely many integral elements, hence $A$ is finite over $B$.
\end{proof}

\begin{remark}
Consider the case $\fg=\sl_2$ and suppose that $V=V^{(2)}\otimes L(2)$ is the sum of several copies of the adjoint module $L(2)$. In this case, the algebra $S_{\le\la}(V)$ for $\la=2$
is precisely the algebra $\mathrm{Com}^{\mathfrak{T}}(V)$ from \cite{dotsenko2023gracestitskantorkoechercategory}, and the
finite-dimensionality of that algebra (in fact, a precise description of its underlying vector space) is established in \cite[Prop.~3.2]{dotsenko2023gracestitskantorkoechercategory}.
\end{remark}

The previous result motivates the following definition.

\begin{definition}
An integrable $\fg$-module $V$ is said to be \emph{thin} if the multiplicities of irreducible $\fg$-modules in it are finite and $[V\col L(0)]=0$.
\end{definition}

\begin{corollary}
If $V$ is a thin $\fg$-module,
then the algebra $S_{\Om}(V)$
is finite-dimensional.
\end{corollary}
\begin{proof}
Consider the decomposition $V=V'\oplus V''$, where $V'=\ta_\Om(V)$ and $V''=\bta_\Om(V)$.
Then $V''\sbs \bta_\Om(SV)$, hence $\bop_{k\ge1}S^k(V'')\sbs S(V'')$ is mapped to zero in $S_{\Om}(V)$.
Therefore the map $S_{\Om}(V')\to S_{\Om}(V)$ is surjective.
By Proposition \ref{prop:com-alg-fd}, the algebra $S_{\Om}(V')$ is finite-dimensional.
\end{proof}

\begin{theorem}\label{main1}
If the adjoint action of $\fg$ on $\bfg$ is thin, then the functor $\Bind_\Om\colon\II(\fg)\to\IIx\bfg{\Om}$ maps finite-dimensional modules to finite-dimensional modules.
\end{theorem}
\begin{proof}
We have $\Bind_\Om (V)=N'/A(\bta_\Om N')$, where $A=U(\bfg)$ and $N'=A\ts_{U(\fg)}V$.
Therefore it is enough to show that
$\bar N=N/A(\bta_\Om N)$ has finite dimension for $N=A\ts V$.
Let $A=\bop_{\mu\in P_+}A^{(\mu)}\ts L(\mu)$.
The image of $A^{(\mu)}\ts L(\mu)\ts V$ in $\bar N$ can be nonzero only if there is a summand $L(\nu)$ of $L(\mu)\ts V$ such that $\nu\in \Om$.
We have
\[\Hom_{\fg}(L(\mu)\ts V,L(\nu))\iso
\Hom_{\fg}(L(\mu),\Hom(V,L(\nu))),
\]
hence $L(\mu)$ has to be a summand of $\Hom(V,L(\nu))$
for some $\nu\in \Om$.
The set $\Om'$ of such $\mu\in P_+$ is finite.
We conclude that $A\bta_{\Om'}(A)\ts V$ is mapped to zero in $\bar N$,
hence $\bar N$ is a quotient of $(A/A\bta_{\Om'} (A))\ts V$
and it is enough to show that the $A$-module $M=A/A(\bta_{\Om'} A)$ is finite-dimensional.
The PBW filtration $(F_i)_{i\ge0}$ on $A$ induces the filtration $(F_iM)_{i\ge0}$ on~$M$.
The epimorphism $A\to M$  induces the epimorphism between associated graded modules $\Gr A\iso S\bfg\to \Gr M$.
We have $\bta_{\Om'} M=0$, hence $\bta_{\Om'}\Gr M=0$ and the above map factorises through $S\bfg\to S_{\Om'}(\bfg)=(S\bfg)/(S\bfg\cdot \bta_{\Om'}(S\bfg))$. According to Proposition \ref{prop:com-alg-fd}, $S_{\Om'}(\bfg)$ is finite-dimensional, completing the proof.
\end{proof}

\subsection{Applications of the main result}\label{sec:applic}

Let us give some examples of situations where our result applies. We begin with a class of examples including that of $L_0(\HH_2)$. Let us denote by $\WW_2$ the Lie algebra of all polynomial vector fields on the plane, and by $L_0(\WW_2)$ the Lie subalgebra of that algebra consisting of all vector fields that vanish at the origin.

\begin{corollary}\label{cor:W2}
For every Lie subalgebra $\bfg\subset L_0(\WW_2)$ that contains
 \[
\fg=\mathbb{C}\{x\partial_y, y\partial_x,x\partial_x-y\partial_y\}\cong \sl_2
 \]
and does not contain the Euler vector field $x\partial_x+y\partial_y$, all global Weyl modules are finite-dimensional.
\end{corollary}

\begin{proof}
A vector field vanishing at the origin is of the form $f(x,y)\partial_x+g(x,y)\partial_y$, where $f$ and $g$ belong to the ideal of $\mathbb{C}[x,y]$ generated by $x$ and $y$. It easily follows that, if we consider $L_0(\WW_2)$ as an $\sl_2$-module with respect to the adjoint action, we have an isomorphism of $\sl_2$-modules
 \[
L_0(\WW_2)\cong L(1)\otimes \bigoplus_{k\ge 1}L(k)\cong \bigoplus_{k\ge 1}(L(k-1)\oplus L(k+1)).
 \]
We note that each irreducible $\sl_2$-module appears in this decomposition with multiplicity at most two. Moreover, this decomposition has exactly one summand $L(0)$; it appears for $k=1$, that is, for linear vector fields, and corresponds precisely to the Euler vector field.
It follows that the adjoint action of $\fg=\sl_2$ on $\bfg$ is thin, and Theorem \ref{main1} applies.
\end{proof}

Another class of examples where our results apply is that where we consider a finite-dimensional Lie algebra containing $\fg$.

\begin{corollary}\label{cor:FD}
For every finite-dimensional Lie algebra $\bfg$ containing $\fg$ as a Lie subalgebra and graded by the root system of $\fg$, if the centralizer of $\fg$ in $\bfg$ is zero, then all global Weyl modules are finite-dimensional.
\end{corollary}

\begin{proof}
The root grading defines the necessary decomposition
$\bfg=\bfn_-\oplus\bfh\oplus\bfn_+$.
Moreover, finite-dimensionality of $\bfg$ together with the zero centralizer condition assures that the adjoint action of $\fg$ on $\bfg$ is thin, so Theorem \ref{main1} applies.
\end{proof}

To give one final example, let us recall the definition of the Lie algebra $\sl(\lambda)$ introduced by Feigin~\cite{MR940679}.
Given $\lambda\in\mathbb{C}$,
consider the quotient of the universal enveloping algebra $U(\sl_2)$ by the two-sided ideal generated by $\Delta-\frac{\la^2-1}2$, where $\De=ef+fe+\frac{h^2}2\in U(\sl_2)$ is the Casimir element.
Let $\gl(\la)$ be the Lie algebra obtained from that associative algebra by considering the usual Lie bracket $[a,b]=ab-ba$.
If $\lambda=n\in\bZ_{\ge1}$, the Casimir element acts by $\frac{n^2-1}2$ on the irreducible $\mathfrak{sl}_2$-module $V=L(n-1)$ of dimension~$n$.
Therefore we have a surjective algebra homomorphism
$\gl(\la)\to \gl(V)\iso\gl_n$,
explaining the name $\gl(\la)$.
Let $\sl(\la)=\gl(\la)/\bC 1$ be the quotient of $\gl(\la)$ by the center.
Being a quotient of $U(\sl_2)$, the Lie algebra $\bfg=\sl(\lambda)$ contains $\fg=\sl_2$ as a Lie subalgebra.

\begin{corollary}\label{cor:sl-lambda}
For the Lie algebra $\bfg=\sl(\lambda)$ and its Lie subalgebra $\fg=\sl_2$, all global Weyl modules are finite-dimensional.
\end{corollary}

\begin{proof}
We have a decomposition with respect to the adjoint action of $\sl_2$
\cite{MR940679}
 \[
\sl(\lambda)=\bigoplus_{n\ge 1}L(2n),
 \]
so the adjoint action of $\fg=\sl_2$ on $\bfg$ is thin, and Theorem \ref{main1} applies.
\end{proof}

\section{Stratifications}\label{sec:stratifications}
In this section we will equip the category $\cA=\IB(\bfg)$ of $\bfg$-modules that are integrable and bounded over $\fg$
with a (left) stratification structure over the poset $P_+$.
It turns out that the strata categories $\cA_\la$ of this stratification (for $\la\in P_+$)
are the categories of left modules over certain algebras $\A_\la$ originally introduced
in \cite{MR2718936}
for the Lie algebras of the form $\bfg=\fg\ts A$, where $A$ is commutative.
These algebras were further studied in \cite{MR3765456} for the root-graded Lie algebras $\bfg$, where it was shown that the algebras $\A_\la$ are isomorphic to the algebras originally introduced in \cite{MR634510} in the context of $\la$-admissible modules.
We will see that $\A_\la^\op$ is isomorphic to the endomorphism algebra of the Weyl module $W(\la)$.
The Weyl functors $\W_\la:\Mod\A_\la\to\II(\bfg)$
(for the algebras $\bfg=\fg\ts A$) have been used to produce and study new $\bfg$-modules,
but their interpretation as building blocks of a stratification structure seems to be missing in the literature.
Using this point of view, we can interpret global Weyl modules as standard objects and local Weyl modules as proper standard objects of our stratified category
(\cf \cite{khoroshkin2015highestweightcategoriesmacdonald}).

\subsection{Left recollement}
If $j^*:\cA\to\cB$ is an exact functor between abelian categories, then the category $\cC=\Ker j^*$ is a Serre subcategory of $\cA$ and the induced functor $\cA/\cC\to\cB$ is faithful and exact \stack{02MN}.
The functor $\cA/\cC\to\cB$ can be full even if $j^*:\cA\to\cB$ was not.

\begin{lemma}\label{lm:recol}
Let $j^*:\cA\to\cB$ be an exact functor between abelian categories
having a fully faithful left adjoint functor $j_!$ (equivalently, the adjunction morphism $\id\to j^*j_!$ is an isomorphism).
Then the induced functor $\cA/\cC\to\cB$, for $\cC=\Ker j^*$,
is an equivalence and the embedding functor $i_*:\cC\to\cA$ has a left adjoint $i^*$.
\end{lemma}

\begin{proof}
\def\jm{\varsigma}
Consider 
the induced functor $\jm^*:\cA/\cC\to\cB$ and
the composition $\jm_!:\cB\xto{j_!}\cA\to\cA/\cC$.
Then $\jm^*\jm_!\iso\id$ and we need to show that $\jm_!\jm^*\to\id$ is an isomorphism
For every $X\in\cA$, the morphism $j_!j^*X\to X$ is an isomorphism in $\cA/\cC$ since $j^*j_!j^*\to j^*X$ is an isomorphism in $\cB$ and $\cA/\cC\to\cB$ is faithful and exact.

For $X\in\cA$, let $i^*(X)=X'=\Coker(j_!j^*X\to X)$.
We have an exact sequence
 \[
j^*j_!j^*X\to j^*X\to j^*X'\to0,
 \]
implying that $j^*X'=0$ and $X'\in\cC$.
For $Y\in\cC$, we have an exact sequence
$0\to\Hom(X',Y)\to\Hom(X,Y)\to\Hom(j_!j^*X,Y)=0$,
hence $\Hom(i^*X,Y)\iso \Hom(X,Y)$.
\end{proof}

A \idef{recollement} of abelian categories (\cf \cite[\S1.4]{BBD})
is a sequence of exact functors
$\cC\xto{i_*}\cA\xto{j^*}\cB$ between abelian categories such that
\begin{enumerate}
\item $j^*i_*=0$ and $i_*:\cC\to\Ker j^*$ is an equivalence.
\item
$j^*$ has left and right fully faithful adjoints $j_!,j_*$.
\end{enumerate}
\[
\begin{tikzcd}
\cC\rar["i_*"{description}]&
\lar[bend right,"i^*"']\lar[bend left,"i^!"]\cA\rar["j^*"{description}]
&\cB\lar[bend left,"j_*"]\lar[bend right,"j_!"']
\end{tikzcd}
\]
By Lemma \ref{lm:recol}, the functor $i_*$ has left and right adjoints $i^*,i^!$
and the induced functor $\cA/\cC\to\cB$ is an equivalence.
The adjunction morphisms
 \[
i^*i_*\to\id\to i^!i_* \quad \text{and} \quad j^*j_*\to\id\to j^*j_!
 \]
are isomorphisms.
We will say that we have a \idef{left recollement} if only the fully faithful left adjoint functor $j_!$ is required to exist.

\begin{example}\label{ex:recol}
Let $(X,\cO)$ be a ringed space, $U\sbs X$ be an open subset, $Z=X\ms U$ and $j:U\emb X$, $i:Z\emb X$ be the embeddings.
Let $\cA=M(X,\cO)$ be the abelian category of (left) $\cO$-modules over $X$ and similarly for $\cB=M(U,\cO)$ and $\cC=M(Z,\cO)$.
Then there is a recollement
\[
\begin{tikzcd}
M(Z,\cO)\rar["i_*"{description}]&
\lar[bend right,"i^*"']\lar[bend left,"i^!"]M(X,\cO)\rar["j^*"{description}]
&M(U,\cO)\lar[bend left,"j_*"]\lar[bend right,"j_!"']
\end{tikzcd}
\]
where
 \[
j^*:M(X,\cO)\to M(U,\cO) \quad \text{and}\quad i^*:M(X,\cO)\to M(Z,\cO)
 \]
are restriction functors,
 \[
j_*:M(U,\cO)\to M(X,\cO) \quad \text{and}\quad i_*:M(Z,\cO)\to M(X,\cO)
 \]
are direct image functors,
$j_!:M(U,\cO)\to M(X,\cO)$ is extension by zero, and
$i^!:M(X,\cO)\to M(Z,\cO)$ is the functor of sections supported on $Z$.
For $F\in M(X,\cO)$ there are exact sequences
$0\to j_!j^*F\to F\to i_*i^*F\to 0$ and $0\to i_*i^!F\to F\to j_*j^*F$.
For an arbitrary recollement we don't have exactness on the left in the first sequence in general (but we have exactness on the right by Lemma \ref{lm:recol}).
\end{example}

Assume that we have a left recollement
$\cC\xto{i_*}\cA\xto{j^*}\cB$.

\begin{lemma}
[\cf\ {\cite[Theorem 1]{MR2718936}}]
An object $X\in\cA$ is contained in $j_!(\cB)$ if and only if
\[\Hom(X,\cC)=\Ext^1(X,\cC)=0.\]
\end{lemma}
\begin{proof}
Let $X=j_!Z$ for $Z\in\cB$ and let $Y\in\cC$.
Then $\Hom(X,Y)=\Hom(Z,j^*Y)=0$.
If $0\to Y\to X'\to X\to 0$ is an exact sequence, then $j^*X'\iso j^*X=j^*j_!Z\iso Z$ and this isomorphism induces $X=j_!Z\to X'$ which splits the above exact sequence.

Conversely, assume that $X\in\cA$ satisfies $\Hom(X,\cC)=\Ext^1(X,\cC)=0$.
We need to show that $j_!j^*X\to X$ is an isomorphism.
Consider an exact sequence $0\to Y_1\to j_!j^*X\to X\to Y_2\to0$.
Applying $j^*$, we obtain an isomorpism $j^*j_!j^*X\isoto j^*X$,
hence $Y_1,Y_2\in\cC$.
As $\Hom(X,\cC)=0$, we obtain $Y_2=0$.
As $\Ext^1(X,\cC)=0$, we obtain $j_!j^*X=Y_1\oplus X$.
But $\Hom(j_!j^*X,Y_1)\iso\Hom(j^*X,j^*Y_1)=0$, hence $Y_1=0$ and
$j_!j^*X\iso X$.
\end{proof}

\begin{lemma}
If $\cB$ has enough projectives,
then the functor
$j_!:\cB\to\cA$ is exact if and only if
$\Ext^2(j_!\cB,\cC)=0$.
\end{lemma}
\begin{proof}
Let $j_!$ be exact.
Consider an exact sequence $0\to X'\to P\to X\to0$
in $\cB$ with projective~$P$.
For $Y\in\cC$, we have an exact sequence
$\Ext^1(j_!X',Y)\to \Ext^2(j_!X,Y)\to \Ext^2(j_!P,Y)$.
The first term is zero by the previous lemma.
The last term is zero since $j_!$ is left adjoint, hence maps projectives to projectives.

Conversely, given an exact sequence $0\to X'\to X\to X''\to0$ in $\cB$, we have an exact sequence
$j_!X'\xto f j_!X\to j_!X''\to0$
and we need to show that $f$ is a monomorphism.
Let $Y_1=\Ker f$ and $Y_2=\Im f$.
We have an exact sequence
 \[
0\to j^*Y_1\to j^*j_!X'\to j^*j_!X,
 \]
hence $j^*Y_1=0$ and $Y_1\in\cC$.
We have an exact sequence
 \[
0=\Ext^1(j_!X',Y_1)\to \Ext^1(Y_2,Y_1)\to \Ext^2(j_!X'',Y_1)=0,
 \]
hence $\Ext^1(Y_2,Y_1)=0$.
Therefore $j_!X'=Y_1\oplus Y_2$.
But $\Hom(j_!X',Y_1)=0$, hence $Y_1=0$.
\end{proof}

\begin{lemma}
If $\cA'\sbs\cA$ is a full subcategory such that $\Hom(\cC,\cA')=\Hom(\cA',\cC)=0$, then
$j^*:\cA'\to\cB$ is fully faithful.
This is true if $\cA'$ is a Serre subcategory such that $\cA'\cap\cC=0$.
\end{lemma}
\begin{proof}
For $X,Y\in\cA$, we have
\[\Hom_{\cA/\cC}(X,Y)
=\dlim\Hom_\cA(X',Y/Y')
\]
where the limit is taken over $X'\sbs X$ with $X/X'\in\cC$ and
$Y'\sbs Y$ with $Y'\in\cC$.
If $X,Y\in\cA'$, then $X/X'=0$ and $Y'=0$, hence
$\Hom_{\cA/\cC}(X,Y)=\Hom_\cA(X,Y)$.
\end{proof}

\begin{lemma}\label{simple corr}
If $L\in\cA\ms\cC$ is simple, then $M=j^*L$ is simple and $L$ is a quotient of $j_!M$.
If $L,L'\in\cA\ms\cC$ are simple and non-isomorphic,
then $j^*L,j^*L'$ are non-isomorphic.
If $\cA$ is of finite length, then there is a 1-1 correspondence between isomorphism classes of simple objects in $\cA\ms\cC$ and isomorphism classes of simple objects in $\cB$.

\end{lemma}
\begin{proof}
For a nonzero subobject $N\emb M=j^*L$, the corresponding map $j_!N\to L$ is nonzero, hence an epimorphism.
Therefore $N\iso j^*j_!N\to j^*L$ is an epimorphism and $N=M$.
This implies that $M$ is simple.
We have seen that $j_!M\to L$ is an epimorphism.
Let $\cA'\sbs\cA$ be the Serre subcategory generated by $L,L'$.
Then $\cA'\cap\cC=0$, hence $\Hom(L,L')\iso\Hom(j^*L,j^*L')$.
If $j^*L,j^*L'$ are isomorphic, then $\Hom(L,L')\ne0$, hence $L,L'$ are isomorphic.

Assume that $\cA$ is of finite length.
For a simple object $M\in\cB$, let $L$ be a simple quotient of $j^!M$.
A nonzero map $j^!M\to L$ corresponds to a nonzero map $M\to j^*L$.
We have seen that $j^*L$ is simple, hence $j^*L\iso M$.
This implies that $j^*$ induces a bijection between the corresponding sets of isomorphism classes.
\end{proof}

\subsection{Left stratifications}
\label{left strat}
We define a \idef{left stratification} of an abelian category $\cA$
by a lower finite poset $\La$ to be a collection of Serre subcategories $(\cA_\Om)_{\Om\sbs\La}$ for finite lower subsets $\Om\sbs\La$ and a collection of abelian categories $\cA_\la$ (called \idef{strata categories}) for $\la\in\La$ such that
the collection $(\cA_\Om)_{\Om\sbs\La}$ is order preserving (if $\Om\sbs\Om'$, then $\cA_{\Om}\sbs\cA_{\Om'}$),
$\cA_\es=0$, $\bigcup_\Om\cA_\Om=\cA$, and for all finite lower subsets $\Om\sbs\La$ and $\la\in\max\Om$ we have compatible left recollements of abelian categories
\[\cA_{\Om\ms\la}\xto{i_*}\cA_\Om\xto{j^*}\cA_\la\]
The above data is called a \idef{stratification} if all sequences are recollements (\cf \cite[\S2.1]{wiggins_stratified}).

\begin{remark}
The above conditions on a left stratification imply that,
for a finite lower subset $\Om\sbs\La$, the embedding $i_*:\cA_\Om\to\cA$ has a left adjoint $i^*:\cA\to\cA_\Om$ (such that $i^*i_*\iso\id$).
For compatibility we require that
$\cA_\la\xto{j_!}\cA_\Om$
is isomorphic to
$\cA_\la\xto{j_!^\la}\cA_{\le\la}\xto{i_*}\cA_\Om$,
for $\la\in\max\Om$.
This implies $i^*j_!\iso j_!^\la$
which is equivalent to $j^*i_*\iso j_\la^*$.
\[\begin{tikzcd}
\cA_{\le\la}\rar[shift right=1,"i_*"']\rar[<-,shift left=1, "i^*"]
\ar[rr,bend right,"j^*_\la"{description}]
\ar[rr,<-,bend left,"j_!^\la"{description}]
&\cA_{\Om}\rar[shift right,"j^*"']\rar[<-,shift left,"j_!"]&\cA_\la
\end{tikzcd}\]
Our notion of a left stratified category is related to the notion of a highest weight category
in~\cite{khoroshkin2015highestweightcategoriesmacdonald}
and a (standardly) stratified category in \cite{MR3338680},
although in these papers one only requires existence of the left adjoint $j^\la_!:\cA_\la\to\cA_{\le\la}$ for the projection
$j^*_\la:\cA_{\le\la}\to\cA_\la:=\cA_{\le\la}/\cA_{<\la}$.
Note that in \cite{MR3338680} the authors require $j^\la_!$
to be exact while we require it to be fully faithful.
\end{remark}

\begin{remark}
If $\la\not\le\mu$, consider $\Om=\La_{\le\la}\cup\La_{\le\mu}$ with $\la\in\max\Om$.
Then $j_!\cA_\mu\sbs\cA_{\Om\ms\la}$, hence
\begin{equation}\label{orthog1}
\Hom(j_!\cA_\la,j_!\cA_\mu)=0,\qquad
\Ext^1(j_!\cA_\la,j_!\cA_\mu)=0,\qquad
\forall \la\not\le\mu.
\end{equation}
\end{remark}

\begin{example}
Let $(X,\cO)$ be a ringed space and $X=\bigsqcup_{\la\in\La}X_\la$ be a partition of $X$ for a finite poset~$\La$ such that the closures of the parts are $\ubar X_\la=\bigsqcup_{\mu\le\la}X_\mu$.
As in Example~\ref{ex:recol},
we consider $\cA=M(X,\cO)$, $\cA_\Om=M(X_\Om,\cO)$ for a lower subset $\Om\sbs\La$ and the closed set $X_\Om=\bigsqcup_{\la\in\Om}X_\la$, and $\cA_\la=M(X_\la,\cO)$.
For $\la\in\max\Om$, the subset $X_\la\sbs X_\Om$ is open in $X_\Om$ and $X_{\Om\ms\la}\sbs X_\Om$ is closed in $X_\Om$.
The recollements
$\cA_{\Om\ms\la}\xto{i_*}\cA_\Om\xto{j^*}\cA_\la$
form a stratification of $\cA$.
Considering $j:X_\la\xto{j^\la}X_{\le\la}\xto i X_\Om$,
we obtain $i_*j_!^\la=i_!j^\la_!=(ij^\la)_!=j_!$
which proves compatibility.
\end{example}

\begin{lemma}
For a simple object $L\in\cA$, there exists a unique $\la\in\La$ such that $L\in\cA_{\le\la}\ms\cA_{<\la}$.
\end{lemma}
\begin{proof}
Let $\Om\sbs\La$ be a minimal lower subset such that $L\in\cA_\Om$ and let $\la\in\max\Om$.
Consider the left recollement
$\cA_{\Om\ms\la}\emb\cA_\Om\xto{j^*}\cA_\la$.
If $j^*L=0$, then $L\in\cA_{\Om\ms\la}$, a contradiction.
Therefore $j^*L\ne0$, hence $j_!j^*L\to L$ is nonzero.
Since $j_!j^*L\in j_!(\cA_\la)\sbs\cA_{\le\la}$ and $\cA_{\le\la}$ is a Serre subcategory, we obtain $L\in\cA_{\le\la}$ and $L\notin\cA_{<\la}$.

If we also have $L\in\cA_{\le\mu}\ms\cA_{<\mu}$, we can assume without loss of generality that $\la\not\le\mu$, hence $\la$ is maximal in $\Om=\La_{\le\la}\cup\La_{\le\mu}$.
Then $j^*:\cA_\Om\to\cA_\la$ satisfies $j^*(L)=0$ since $\cA_{\le\mu}\sbs\cA_{\Om\ms\la}$.
Therefore $j_\la^*L=0$ and $L\in\cA_{<\la}$, a contradiction.
\end{proof}

\oper{Irr}
Let $\Irr(\cA)$ denote the set of isomorphism classes of simples objects in $\cA$.
Then we have a commutative diagram
\begin{equation}\label{Irr diagram}
\begin{tikzcd}
\Irr(\cA)\ar[rr,"\vi"]\ar[dr,"\rho"']&&\bigsqcup_{\la\in\La}\Irr(\cA_\la)\ar[dl]\\
&\La
\end{tikzcd}
\end{equation}
where, for $L\in\Irr(\cA)$ contained in $\cA_{\le\la}\ms\cA_{<\la}$,
we define $\rho(L)=\la$ and $\vi(L)=j^*_\la(L)\in\Irr(\cA_\la)$.
It follows from Lemma \ref{simple corr} that the map $\vi$ is injective and that it is bijective if $\cA$ is of finite length.
Let $\Te$ be a set parameterizing isomorphism classes of simple objects in $\cA$ (so that we have a bijection $L:\Te\to\Irr(\cA)$).
For $\te\in\Te$, we denote $\la=\rho L(\te)$ by $\rho(\te)$ and we denote $\vi L(\te)\in\cA_\la$ by $L_\la(\te)$.
We define the \idef{proper standard object}
(\cf \cite{khoroshkin2015highestweightcategoriesmacdonald,MR3338680,wiggins_stratified})
\[\ubar\De(\te)=j_! L_\la(\te)\in\cA_{\le\la}.\]
Assuming that $L_\la(\te)\in\cA_\la$ has a projective cover $P_\la(\te)\in\cA_\la$, we define the \idef{standard object}
\[\De(\te)=j_!P_\la(\te)\in\cA_{\le\la}.\]
We say that a left stratified category
is \idef{standardly stratified}
if the projective cover $P(\te)$ of $L(\te)\in\cA$ has a filtration by standard objects such that the top factor is $\De(\te)$ and all other factors $\De(\te')$ satisfy $\rho(\te')>\rho(\te)$.
Note that under certain mild restrictions, every stratified category is standardly stratified
\cite[\S2.4]{wiggins_stratified}, but this will be not our case.

\subsection{The algebra \texorpdfstring{$\A_\la$}{Alambda}}
\label{sec:Weylfunctor}



\begin{lemma}
There is a uniquely defined surjective algebra homomorphism
\[\vi:U(\bfh)^\op\to\End_{\bfg}(W(\la))\]
such that $\vi_a(v_\la)=av_\la$ for $a\in U(\bfh)$.
\end{lemma}
\begin{proof}
We have $\End_{\bfg}(W(\la))\iso W(\la)_\la$
by Lemma \ref{lm:weyl prop}.
For every $a\in U(\bfh)$ we obtain
$\vi_a\in\End_{\bfg}(W(\la))$ specified by
$\vi_a(v_\la)=av_\la\in W(\la)_\la$.
For $a,b\in U(\bfh)$, we have
\[\vi_a\vi_b(v_\la)=\vi_a(bv_\la)=bav_\la=\vi_{ba}(v_\la),\]
hence $\vi:U(\bfh)^\op\to\End_{\bfg}(W(\la))$
is an algebra homomorphism.
The composition
 \[
U(\bfh)\xto\vi \End_{\bfg}(W(\la))\isoto W(\la)_\la
 \]
is given by $a\mto av_\la$ which is a surjective map.
Therefore $\vi$ is also surjective.
\end{proof}

We consider the ideal
$\sJ_\la=\Ker\vi
=\Ann_{U(\bfh)}(v_\la)$
and define the algebra
 \[
\A_\la:=U(\bfh)/\sJ_\la\iso\End_{\bfg}(W(\la))^\op.
\]

We see that the Weyl module $W(\la)$ is a $U(\bfg)$-$\A_\la$-bimodule
(with $(uv_\la)a=\vi_a(uv_\la)=uav_\la$ for $u\in U(\bfg)$ and $a\in\A_\la$)
and each weight space $W(\la)_\mu$ is an $\A_\la$-bimodule.
Moreover,
we have an isomorphism of $\A_\la$-bimodules
 \[
\A_\la\isoto W(\la)_\la,\qquad a\mto av_\la.
 \]
Let $\Mod\A_\la$ denote the category of left modules over $\A_\la$.
By analogy with \cite[\S3.4]{MR2718936}, we define the \idef{Weyl functor}
\[
\W_\la\colon \Mod \A_\la\to \IIx\bfg{\le\la},\qquad
M\mto W(\la)\otimes_{\A_\la}M.
 \]

The following result generalizes \cite[\S3]{MR2718936}.

\begin{lemma}
The functor $\W_\la:\Mod \A_\la\to \IIx\bfg{\le\la}$
is left adjoint to the functor
\[\HW_\la:\IIx\bfg{\le\la}\to\Mod\A_\la,\qquad
V\mto V_\la \iso\Hom_{\bfg}(W(\la),V).\]
The adjunction morphism $\id\to\R_\la\W_\la$ is an isomorphism,
meaning that $\W_\la$ is fully faithful.
\end{lemma}
\begin{proof}
By Lemma \ref{lm:weyl prop}, we have
$\HW_\la(V)=V_\la\iso \Hom_{\bfg}(W(\la),V)$
and it is a left module over $\A_\la\iso\End_{\bfg}(W(\la))^\op$.
Explicitly, for $a\in\A_\la$ and $f\in\Hom_\bfg(W(\la),V)$
with $v=f(v_\la)$, we have
$(af)(v_\la)=(f\vi_a)(v_\la)=f(av_\la)=av$.
The claim on the adjunction follows from
$\Hom_{\bfg}(\W_\la(M),V)
\iso\Hom_{\A_\la}(M,\Hom_{\bfg}(W(\la),V))
\iso\Hom_{\A_\la}(M,\HW_\la(V))$.
We have $\R_\la\W_\la(M)=(W(\la)\ts_{\A_\la}M)_\la
=W(\la)_\la\ts_{\A_\la}M\iso M$ since $W(\la)_\la\iso\A_\la$.
\end{proof}

\begin{remark}
If $V\in\II(\bfg)$ and $\la\in \max P(V)$, then
$\Hom_{\bar\fg}(W(\la),V)\iso\Hom_{\fg}(L(\la),V)\iso V_\la$.
Therefore, in the same way as above,
we can prove for a finite lower subset $\Om\sbs P_+$
and $\la\in\max\Om$, that the functor
$\W_\la:\Mod\A_\la\to\IIx\bfg\Om$, $M\mto W(\la)\ts_{\A_\la}M$,
is left adjoint to the functor
$\R_\la:\IIx\bfg\Om\to\Mod\A_\la$, $V\mto V_\la$.
\end{remark}

Using the experimental data presented in Section \ref{sec:globalH2}, one can obtain some information about the algebras $\A_\la$ in the case $\bfg=L_0(\HH_2)$. In particular, it is easy to see that, contrary to the case of affine algebras \cite{MR1850556},
the global Weyl module $W(\la)$ is generally not a free $\A_\la$-module.

\begin{example}
Let us consider the global Weyl module $W(4)$ for the algebra $\bfg=L_0(\HH_2)$. As indicated in Section \ref{sec:globalH2}, this is a finite-dimensional vector space of dimension $31$. Since the multiplicity of weight $4$ in this module is equal to $2$, the algebra $\A_4$ is two-dimensional, and the $\A_4$-module $W(4)$ is not free.
\end{example}

\begin{theorem}\label{th:strat1}
The category $\cA=\IB(\bfg)$ is left stratified by the poset $\La=P_+$ with the Serre subcategories $\cA_\Om=\II_\Om(\bfg)$ for finite lower subsets $\Om\sbs\La$, the strata categories $\cA_\la=\Mod\A_\la$ for $\la\in\La$,
and the left recollement for $\la\in\max\Om$ (where $i_*$ is the embedding)
\[\cA_{\Om\ms\la}\xto{i_*}\cA_\Om\xto{j^*=\R_\la}\cA_\la=\Mod\A_\la\]
\end{theorem}
\begin{proof}
We have seen that $\R_\la$ has the left adjoint $j_!=\W_\la$ which is fully faithful.
The functor $i_*$ has the left adjoint $i^*=\ta_{\Om\ms\la}^{\bfg}$ \eqref{trunc2}.
The category $\cA_{\Om\ms\la}$ can be identified with the subcategory of objects $V\in\cA_\Om$ such that $\R_\la(V)=V_\la=0$.
\end{proof}

\subsection{Graded case}
\label{sec:graded}
Let $\bfg=\bop_{i\ge0}\bfg\gr i$ be a graded Lie algebra with the semisimple degree zero component $\fg=\bfg\gr 0$.
We say that a $\bfg$-module $V$ is $\bZ$-graded if it is equipped with a grading $V=\bop_{i\in\bZ}V\gr i$  compatible with the grading of $\bfg$, meaning that $\bfg\gr iV\gr j\sbs V\gr{i+j}$ for all $i,j$.
We define the shifted module $V[n]$ with $V[n]\gr i=V\gr{i-n}$.
Let $\Ig(\bfg)$ denote the category of $\bZ$-graded $\bfg$-modules that are $\fg$-integrable (with morphisms of degree zero)
and let $\cA=\IBg(\bfg)\sbs\Ig(\bfg)$ be the subcategory of modules that are bounded as $\fg$-modules.
For a finite (lower) subset $\Om\sbs P_+$, let $\cA_\Om=\Ig_\Om(\bfg)\sbs\cA$ consist of $V\in\IBg(\bfg)$ such that $V=\ta_\Om^\bfg(V)$.


The simple objects of $\Ig(\bfg)$ are objects $L(\la,n)=L(\la)[n]$ indexed by pairs $(\la,n)\in \Te=P_+\xx\bZ$;
such an object is simply $L(\la)$ placed in degree~$n$, on which $\bfg$ acts via the pullback along the projection $\pi\colon\bfg\to\fg\gr0$.
The functor $W_\Om$ defined in \eqref{W_Om} can be lifted to $W_\Om:\Ig(\fg)\to\Ig_\Om(\bfg)$.
In particular, the global Weyl module
$W(\la)=\Bind_{\le\la}(L(\la))\in\Ig_{\le\la}(\bfg)$
is a graded $\bfg$-module (with the degree zero component equal to  $L(\la)$).
We define $W(\la,n)=W(\la)[n]$ for $n\in\bZ$.

\begin{example}
At this point, we may already see that the category $\IBg(\bfg)$ is not generally a standardly stratified category.
If that were the case, it would imply that projective modules $P(\la)$ have a filtration by standard modules, which are exactly the global Weyl modules.
Let $\bfg=L_0(\HH_2)$. We have $P(\la)\cong \Ind_{\sl_2}^{\bfg}L(\la)$, and using this formula, we can immediately compute the first few graded components of $P(0)$:
 \[
P(0)=\begin{tabular}{c}
$L(0)$\\
$L(3)$\\
$L(6)\oplus L(4)\oplus L(2)$\\
\ldots
\end{tabular}
 \]
If the module $P(0)$ had a filtration by standard modules, the above decomposition would have a copy of~$W(3)$ with degrees shifted by one, and examining the decomposition of~$W(3)$ above, we see that this would imply presence of $L(0)$ in the degree $2$ component of~$P(0)$.
The resulting contradiction shows that the category $\IBg(L_0(\HH_2))$ is not standardly stratified.
\end{example}

The algebra $\A_\la\iso \lb\End_\bfg(W(\la))^\op$ is $\bZ$-graded (here we consider endomorphisms of all degrees).
Its degree zero component is $U(\fh)/(h-\la(h)\col h\in\fh)\iso\bC$.
Let $\cA_\la=\Modg(\A_\la)$ be the category of $\bZ$-graded left $\A_\la$-modules.
Its simple objects are of the form $\bC_\la[n]$, for $n\in\bZ$,
where $\bC_\la$ is the module corresponding to the projection $\A_\la\to(\A_\la)_0=\bC$.
Similarly to Theorem \ref{th:strat1} we obtain the following result.

\begin{theorem}\label{th:strat2}
The category $\cA=\IBg(\bfg)$ is left stratified by the poset $\La=P_+$ with the Serre subcategories $\cA_\Om=\Ig_\Om(\bfg)$ for finite lower subsets $\Om\sbs\La$, the strata categories $\cA_\la=\Modg(\A_\la)$ for $\la\in\La$,
and the left recollement for $\la\in\max\Om$ (where $i_*$ is the embedding)
\[\cA_{\Om\ms\la}\xto{i_*}\cA_\Om\xto{j^*=\R_\la}\cA_\la=\Modg(\A_\la).\]
\end{theorem}

The functor $\R_\la:\cA_{\le\la}\to\cA_\la$ induces a 1-1 correspondence between isomorphism classes of simple objects
$L(\la,n)\in\cA_{\le\la}\ms\cA_{<\la}$
and isomorphism classes of simple objects $\bC_\la[n]\in\cA_\la$.
We obtain a commutative diagram (\cf \eqref{Irr diagram})
\begin{equation}
\begin{tikzcd}
\Te=\La\xx\bZ\ar[rr,"\vi"]\ar[dr,"\rho"']&&\bigsqcup_{\la\in\La}\Irr(\cA_\la)\ar[dl]\\
&\La
\end{tikzcd}
\end{equation}
where $\vi$ is the bijection that maps $(\la,n)\in\Te$ to the irreducible module
$\R_\la L(\la,n)=\bC_\la[n]\in\cA_\la$
and $\rho(\la,n)=\la$.
The projective cover of $\bC_\la[n]$ is $\A_\la[n]$, hence the corresponding standard object is
\begin{equation}
\De(\la,n)=\W_\la(\A_\la[n])=W(\la)[n],
\end{equation}
the shifted global Weyl module.
The corresponding proper standard object is
\begin{equation}\label{prop stand graded}
\ubar\De(\la,n)=\W_\la(\bC_\la[n])=\W_\la(\bC_\la)[n].
\end{equation}
Later we will identify these objects with local Weyl modules.

\subsection{Full stratification}
Let $\bfg$ be a graded Lie algebra as before, with finite-dimensional graded components.
Let $\modg(\bfg)$ be the category of $\bZ$-graded $\bfg$-modules that are finite-dimensional at every degree
and let $\cA=\modBg(\bfg)=\modg(\bfg)\cap\IBg(\bfg)$ and $\cA_\Om=\modg_\Om(\bfg)=\modg(\bfg)\cap\Ig_\Om(\bfg)$
for finite lower subsets $\Om\sbs \La=P_+$.
Similarly, let $\cA_\la=\modg(\A_\la)$ be the category of $\bZ$-graded $\A_\la$-modules that are finite-dimensional at every degree.
We have $W(\la)\in\modg_{\le\la}(\bfg)$ and the induced functor $\W_\la:\modg(\A_\la)\to\modg_{\le\la}(\bfg)$ which is left adjoint to the functor $\R_\la:\modg_{\le\la}(\bfg)\to \modg(\A_\la)$.
We would like to construct the right adjoint of the functor~$\R_\la$ which will be a part of the (full) stratification structure on the category $\cA=\modBg(\bfg)$. Specifically, we shall now show that this can be done if there exists a (degree-preserving) Lie algebra automorphism $\si:\bfg\to\bfg$  such that $\si|_{\fh}=-\id$. The restriction of such automorphism $\si$ to $\fg=\bfg\gr0$ sends every root space $\fg_\alpha$ to $\fg_{-\alpha}$, so, up to a choice of a Chevalley basis, it is given by the Chevalley involution of $\fg$, so that $\si(e_i)=-f_i$, $\si(f_i)=-e_i$ and $\si(h_i)=-h_i$.

\begin{example}
As in Example \ref{H2}, let $\bfg=L_0(\HH_2)$
be the Lie algebra of Hamiltonian vector fields of degree $\ge0$.
We consider the automorphism of $\bfg$ given by
$\si(f\dd_x+g\dd_y)=f(y,-x)\dd_y-g(y,-x)\dd_x$
(corresponding to the symplectomorphism of $\bC^2$ given by $(x,y)\mto(y,-x)$).
Recall that $\sl_2\iso\bfg\gr0$ with $e=x\dd_y$, $f=y\dd_x$, $h=x\dd_x-y\dd_y$. Note that though the restriction to $\sl_2$ is the Chevalley involution, on the whole Lie algebra $\bfg=L_0(\HH_2)$ we only have $\si^4=\id$.
\end{example}


In what follows, for any algebra homomorphism $\vi:A\to B$, we denote the induced functor $\Mod B\to\Mod A$ again by $\vi$.
Consider the duality functor
\[D:\modg(\bfg)^\op\to\modg(\bfg),\qquad
(DV)\gr i=V\gr {-i}^*,\]
and similarly $D:\modg(\A_\la)^\op\to \modg(\A_\la^\op)$.
Note that $(DV)_\la=V_{-\la}^*$ for $V\in \modg(\bfg)$ and $\la\in P$.
Taking the pullback along $\si$, we obtain the module
$\bar D V=D\si V$ such that $(\bar D V)_\la=V_\la^*$.
This implies that $\bar D$ restricts to an equivalence
\[\bar D:\modg_\Om(\bfg)^\op\to\modg_\Om(\bfg)\]


\begin{lemma}
For a finite lower subset $\Om\sbs \La$, $\la\in\max\Om$ and
$V\in\modg_{\Om}(\bfg)$,
we have an isomorphism of $U(\bfh)$-modules
\[D\R_\la\bar D(V)\iso \si S\R_\la(V)\]
where $S:U(\bfh)^\op\to U(\bfh)$ is the antipode (defined by $S(a)=-a$ for $a\in\bfh$).
In particular, the ideal $\sJ_\la=\Ker(U(\bfh)\to\A_\la)$ is stable under the action of $\si S$.
\end{lemma}
\begin{proof}
The left action of $U(\bfh)$ on $\R_\la \bar D V=V_\la^*$ is given by $(af)(v)=-f(\si(a)v)$ for $a\in \bfh$, $f\in V_\la^*$ and $v\in V_\la$.
Therefore the right action of $U(\bfh)$ on $D\R_\la \bar DV=V_\la$
is given by $va=-\si(a)v$ for $a\in\bfh$.
This coincides with the right action of $U(\bfh)$ on $\si S\R_\la(V)$.
The last assertion follows from the fact that $D\R_\la\bar D(V)$ is a module over $\A_\la$.
\end{proof}


The above result implies that we have an equivalence
\[\hat D=D\si S:\modg(\A_\la)^\op\to \modg(\A_\la)\]
such that $\R_\la\bar D\iso \hat D\R_\la$.

\begin{lemma}
For a finite lower subset $\Om\sbs \La$ and $\la\in\max\Om$,
the functor
 \[
\R_\la:\modg_{\Om}(\bfg)\to\modg(\A_\la)
 \]
has the
fully faithful right adjoint $\bar D\inv \W_\la\hat D$.
\end{lemma}

\begin{proof}
For $M\in\modg (\A_\la)$, $V\in\modg_{\Om}(\bfg)$
we have
\begin{multline*}
\Hom(\R_\la V, M)
\iso\Hom(\hat DM, \hat D\R_\la V)
\iso\Hom(\hat DM,\R_\la \bar D V)
\\ \iso\Hom(\W_\la \hat D M, \bar D V)
\iso \Hom(V,\bar D\inv \W_\la \hat D M).
\end{multline*}
\end{proof}


This result together with Theorem \ref{th:strat2} implies

\begin{theorem}
For a Lie algebra $\bfg$ that has a degree-preserving automorphism $\si$ such that $\si|_{\fh}=-\id$, the category $\cA=\modBg(\bfg)$ is stratified by the poset $\La=P_+$ with the Serre subcategories $\cA_\Om=\modg_\Om(\bfg)$ for finite lower subsets $\Om\sbs\La$, the strata categories $\cA_\la=\modg(\A_\la)$ for $\la\in\La$,
and the recollement for $\la\in\max\Om$ (where $i_*$ is the embedding)
\[\cA_{\Om\ms\la}\xto{i_*}\cA_\Om\xto{j^*=\R_\la}\cA_\la=\modg(\A_\la).\]
\end{theorem}


\section{Local Weyl modules}\label{sec:local}

In the case of affine algebras, map algebras, and other root-graded algebras, there is an important notion of a local Weyl module; for instance, for algebras $\bfg$ of the form $\fg\otimes A$ with $A$ commutative, local Weyl modules help to capture aspects of representation theory of $\bfg$ at different points of $\Spec(A)$. In this section, we discuss a general version of this definition
and relate local Weyl modules to proper standard modules introduced earlier.

\subsection{Definition and basic properties}

\begin{definition}\label{loc1}
For $z\in(\bfh)^*$, the \idef{local Weyl module} $W(z)$ is the maximal $\fg$-integrable $\bfg$-module generated by the vector $v_z$ such that
\[\bfn_+v_z=0,\qquad hv_z=z(h)v_z\quad\forall h\in\bfh. \]
\end{definition}

The standard example of a local Weyl module
is as follows (see \eg \cite[Def.~4]{MR2102326}).

\begin{example}\label{ex1}
Let $X=\Spec A$ be an affine scheme over $\bC$ and $\la=\sum_{i=1}^r\la_i\om_i$ be a dominant weight.
Let $\bfg=\fg\ts A$ so that $\bfh=\fh\ts A$.
Let $(z_{ij})_{1\le j\le\la_i}$ be a collection of points in $X$.
We can interpret them as algebra homomorphisms $z_{ij}\colon A\to\bC$.
Consider the linear maps $z_i=\sum_{j=1}^{\la_i}z_{ij}\colon A\to\bC$ and the linear map
 \[
z\colon\bfh=\fh\ts A\to\bC,\qquad h_i\ts a\mto z_i(a).
\]
We have $z(h_i)=z(h_i\ts 1)=\sum_{j=1}^{\la_i} z_{ij}(1)=\la_i$, hence $z|_{\fh}=\la$.
The local Weyl module $W(z)$ is the
maximal $\fg$-integrable $\bfg$-module generated by the vector $v_z$ such that
 \[
(\fn_+\ts A)v_z=0,\qquad (h_i\ts a)v_z=\sum_{j=1}^{\la_i}z_{ij}(a)\cdot v_z\quad \forall 1\le i\le r,\,a\in A.
 \]
\end{example}

The local Weyl modules can be computed using the functor $\W_\la$ of Section \ref{sec:Weylfunctor}.


\begin{lemma}
\label{lm:loc}
Let $z\in(\bfh)^*$ and $\la=z|_{\fh}\in\fh^*$. If $W(z)\ne0$, then $\la$ is a dominant weight. Moreover, we have
 \[
W(z)\cong \W_\la(\bC_z),
 \]
where $\bC_z$ is the one-dimensional $\bfh$-module corresponding to $z$.
\end{lemma}
\begin{proof}
Note that $hv_z=\la(h)v_z$ for $h\in\fh$, so there is a surjective $\bfg$-module morphism $W(\la)\to W(z)$ sending $v_\la$ to $v_z$. Since $W(\la)$ is non-zero only for dominant $\la$, the first assertion follows.
We have $W(z)_\la\iso \bC_z$ as a module over $U(\bfh)$ (and $\A_\la$).
This isomorphism induces a surjective map $\W_\la(\bC_z)\to W(z)$.
The generator of $\W_\la(\bC_z)$ satisfies the conditions of a local Weyl module, hence $\W_\la(\bC_z)\to W(z)$ is an isomorphism by the maximality of $W(z)$.
\end{proof}

The main class of examples of local Weyl modules we
are interested in is as follows.

\begin{definition}
Let $\la$ be a dominant weight and $\pi\colon \bfg\to\fg$ be a left inverse of the embedding of Lie algebras $\fg\emb\bfg$ such that $\pi(\bfh)=\fh$.
The \emph{local augmentation Weyl module} $W_{\pi}(\la)$ is the local Weyl module corresponding to the linear function $\la\pi\colon\bfh\to\bC$.
\end{definition}


\begin{example}
For a point $p\in X=\Spec A$ (identified with the morphism $p\colon A\to\bC$), we may consider $\pi=\id\ts p\colon\bfg=\fg\ts A\to\fg\ts\bC=\fg$.
The corresponding local augmentation Weyl module for $\la=\sum_i\la_i\om_i$ is precisely the module from Example \ref{ex1} with $z_{ij}=p$ for all $1\le j\le\la_i$.
Indeed, we have
 \[
z(h_i\ts a)=z_i(a)=\la_i p(a)=(\la\ts p)(h_i\ts a)
=\la\pi(h_i\ts a)
 \]
for $a\in A$.
\end{example}

For a dominant weight $\la$, the irreducible $\fg$-module $L(\la)$ can be made into a $\bfg$-module using the pullback along $\pi\colon\bfg\to\fg$.
This module satisfies all the conditions of $W_{\pi}(\la)$ except maximality. Therefore, local augmentation Weyl modules are always nonzero.

\medskip
Let us consider now the situation of Section \ref{sec:graded}, where we have extra grading. Using the projection $\pi\colon\bfg\to\bfg\gr0=\fg$, we may define the corresponding local augmentation Weyl modules $W_\pi(\la)$.
By Definition~\ref{loc1} this is the maximal $\fg$-integrable $\bfg$-module generated by a vector $v_z$ such that
$\bfn_+v_z=0$, $hv_z=\la(h)v_z$ for $h\in\fh$ and $\bfh\gr iv_z=0$ for $i>0$,
where we consider the decomposition $\bfh=\bop_{i\ge 0}\bfh_i$ with $\bfh\gr0=\fh$.
By Lemma \ref{lm:loc}, we have
$W_\pi(\la)=W(\la\pi)=\W_\la(\bC_{\la\pi})$
and this is exactly the proper standard object $\ubar\De(\la,0)\in\IBg(\bfg)$
introduced in \eqref{prop stand graded}.

\subsection{Local Weyl modules for \texorpdfstring{$L_0(\HH_2)$}{L0H2}}

The following are $\mathfrak{sl}_2$-decompositions of local augmentation Weyl modules for the Lie algebra $L_0(\HH_2)$ for $\lambda=0,\ldots,7$, listed degree by degree.

 \[
W_\pi(0)=
\begin{tabular}{c}
$L(0)$
\end{tabular},\quad
W_\pi(1)=\begin{tabular}{c}
$L(1)$
\end{tabular},\quad
W_\pi(2)=\begin{tabular}{c}
$L(2)$\\
$L(1)$
\end{tabular},\quad
W_\pi(3)=\begin{tabular}{c}
$L(3)$\\
$L(2)\oplus L(0)$\\
$L(1)$
\end{tabular},
 \]

 \[
W_\pi(4)=\begin{tabular}{c}
$L(4)$\\
$L(3)\oplus L(1)$\\
$2L(2)\oplus L(0)$\\
$L(1)$
\end{tabular},\quad
W_\pi(5)=\begin{tabular}{c}
$L(5)$\\
$L(4)\oplus L(2)$\\
$2L(3)\oplus 2L(1)$\\
$L(4)\oplus 2L(2)\oplus 2L(0)$\\
$L(1)$
\end{tabular},
 \]

 \[
W_\pi(6)=\begin{tabular}{c}
$L(6)$\\
$L(5)\oplus L(3)$\\
$2L(4)\oplus 2L(2)\oplus L(0)$\\
$L(5)\oplus 3L(3)\oplus 3L(1)$\\
$L(4)\oplus 4L(2)\oplus 2L(0)$\\
$L(1)$
\end{tabular},\quad
W_\pi(7)=\begin{tabular}{c}
$L(7)$\\
$L(6)\oplus L(4)$\\
$2L(5)\oplus 2L(3)\oplus L(1)$\\
$L(6)\oplus 3L(4)\oplus 4L(2)\oplus L(0)$\\
$2L(5)\oplus 5L(3)\oplus 5L(1)$\\
$2L(4)\oplus 5L(2)\oplus 4L(0)$\\
$L(1)$
\end{tabular}
 \]

A surprising phenomenon that one observes while examining this data is that the socle of $W_\pi(n)$ seems to be isomorphic to $L(1)$ for all $n\ge 1$. We conjecture that this is indeed the case.







\begin{thebibliography}{10}

\bibitem{MR3951769}
Irfan Bagci, Lucas Calixto, and Tiago Macedo, \emph{Weyl modules and {W}eyl
  functors for {L}ie superalgebras}, Algebr. Represent. Theory \textbf{22}
  (2019), no.~3, 723--756.

\bibitem{BBD}
A.~A. Beilinson, J.~Bernstein, and P.~Deligne, \emph{{F}aisceaux pervers},
  Analysis and topology on singular spaces, I (Luminy, 1981), Ast\'erisque,
  vol. 100, Soc. Math. France, Paris, 1982, pp.~5--171.

\bibitem{MR1161095}
S.~Berman and R.~V. Moody, \emph{Lie algebras graded by finite root systems and
  the intersection matrix algebras of {S}lodowy}, Invent. Math. \textbf{108}
  (1992), no.~2, 323--347.

\bibitem{MR1484478}
Wieb Bosma, John Cannon, and Catherine Playoust, \emph{The {M}agma algebra
  system. {I}. {T}he user language}, J. Symbolic Comput. \textbf{24} (1997),
  no.~3-4, 235--265, Computational algebra and number theory (London, 1993).

\bibitem{MR3981101}
Lucas Calixto, Joel Lemay, and Alistair Savage, \emph{Weyl modules for {L}ie
  superalgebras}, Proc. Amer. Math. Soc. \textbf{147} (2019), no.~8,
  3191--3207.

\bibitem{MR2718936}
Vyjayanthi Chari, Ghislain Fourier, and Tanusree Khandai, \emph{A categorical
  approach to {W}eyl modules}, Transform. Groups \textbf{15} (2010), no.~3,
  517--549.

\bibitem{MR2423816}
Vyjayanthi Chari, Ghislain Fourier, and Prasad Senesi, \emph{Weyl modules for
  the twisted loop algebras}, J. Algebra \textbf{319} (2008), no.~12,
  5016--5038.

\bibitem{MR3848450}
Vyjayanthi Chari, Deniz Kus, and Matt Odell, \emph{Borel--de {S}iebenthal
  pairs, global {W}eyl modules and {S}tanley-{R}eisner rings}, Math. Z.
  \textbf{290} (2018), no.~1-2, 649--681.

\bibitem{MR2271991}
Vyjayanthi Chari and Sergei Loktev, \emph{Weyl, {D}emazure and fusion modules
  for the current algebra of {$\mathfrak{sl}_{r+1}$}}, Adv. Math. \textbf{207}
  (2006), no.~2, 928--960.

\bibitem{MR1850556}
Vyjayanthi Chari and Andrew Pressley, \emph{Weyl modules for classical and
  quantum affine algebras}, Represent. Theory \textbf{5} (2001), 191--223.

\bibitem{MR957457}
E.~Cline, B.~Parshall, and L.~Scott, \emph{Algebraic stratification in
  representation categories}, J. Algebra \textbf{117} (1988), no.~2, 504--521.

\bibitem{cline_stratifying}
Edward Cline, Brian Parshall, and Leonard Scott, \emph{Stratifying endomorphism
  algebras}, Mem. Amer. Math. Soc. \textbf{124} (1996), no.~591, viii+119.

\bibitem{dotsenko2023gracestitskantorkoechercategory}
Vladimir Dotsenko and Iryna Kashuba, \emph{The three graces in the
  {T}its--{K}antor--{K}oecher category}, 2023,
  \href{http://arxiv.org/abs/2310.20635}{{\ttfamily arXiv:2310.20635}}.

\bibitem{MR3520262}
S.~Eswara~Rao, V.~Futorny, and Sachin~S. Sharma, \emph{Weyl modules associated
  to {K}ac-{M}oody {L}ie algebras}, Comm. Algebra \textbf{44} (2016), no.~12,
  5045--5057.

\bibitem{MR940679}
Boris Feigin, \emph{Lie algebras {${\rm gl}(\lambda)$} and cohomology of a
  {L}ie algebra of differential operators}, Uspekhi Mat. Nauk \textbf{43}
  (1988), no.~2(260), 157--158.

\bibitem{MR2102326}
Boris Feigin and Sergey Loktev, \emph{Multi-dimensional {W}eyl modules and
  symmetric functions}, Comm. Math. Phys. \textbf{251} (2004), no.~3, 427--445.

\bibitem{feigin2023peterweyltheoremiwahorigroups}
Evgeny Feigin, Anton Khoroshkin, Ievgen Makedonskyi, and Daniel Orr,
  \emph{{P}eter--{W}eyl theorem for {I}wahori groups and highest weight
  categories}, 2023, \href{http://arxiv.org/abs/2307.02124}{{\ttfamily
  arXiv:2307.02124}}.

\bibitem{MR3703467}
Evgeny Feigin and Ievgen Makedonskyi, \emph{Generalized {W}eyl modules, alcove
  paths and {M}acdonald polynomials}, Selecta Math. (N.S.) \textbf{23} (2017),
  no.~4, 2863--2897.

\bibitem{MR3682816}
\bysame, \emph{Generalized {W}eyl modules for twisted current algebras},
  Teoret. Mat. Fiz. \textbf{192} (2017), no.~2, 284--306.

\bibitem{MR3787562}
Evgeny Feigin, Ievgen Makedonskyi, and Daniel Orr, \emph{Generalized {W}eyl
  modules and nonsymmetric {$q$}-{W}hittaker functions}, Adv. Math.
  \textbf{330} (2018), 997--1033.

\bibitem{MR3055822}
Ghislain Fourier, Nathan Manning, and Prasad Senesi, \emph{Global {W}eyl
  modules for the twisted loop algebra}, Abh. Math. Semin. Univ. Hambg.
  \textbf{83} (2013), no.~1, 53--82.

\bibitem{gorsky2024tautologicalclassessymmetrykhovanovrozansky}
Eugene Gorsky, Matthew Hogancamp, and Anton Mellit, \emph{Tautological classes
  and symmetry in {K}hovanov--{R}ozansky homology}, 2024,
  \href{http://arxiv.org/abs/2103.01212}{{\ttfamily arXiv:2103.01212}}.

\bibitem{MR1918676}
Mark Haiman, \emph{Vanishing theorems and character formulas for the {H}ilbert
  scheme of points in the plane}, Invent. Math. \textbf{149} (2002), no.~2,
  371--407.

\bibitem{hausel2022pwh2}
Tamas Hausel, Anton Mellit, Alexandre Minets, and Olivier Schiffmann,
  \emph{{$P=W$} via {$H_2$}}, 2022,
  \href{http://arxiv.org/abs/2209.05429}{{\ttfamily arXiv:2209.05429}}.

\bibitem{MR1104219}
Victor~G. Kac, \emph{Infinite-dimensional {L}ie algebras}, Cambridge University
  Press, Cambridge, 1990.

\bibitem{khoroshkin2015highestweightcategoriesmacdonald}
Anton Khoroshkin, \emph{Highest weight categories and {M}acdonald polynomials},
  2015, \href{http://arxiv.org/abs/1312.7053}{{\ttfamily arXiv:1312.7053}}.

\bibitem{MR4160929}
Ryosuke Kodera, \emph{Level one {W}eyl modules for toroidal {L}ie algebras},
  Lett. Math. Phys. \textbf{110} (2020), no.~11, 3053--3080.

\bibitem{lau2023jordanalgebrasweightmodules}
Michael Lau and Olivier Mathieu, \emph{Jordan algebras and weight modules},
  2023, \href{http://arxiv.org/abs/2312.16766}{{\ttfamily arXiv:2312.16766}}.

\bibitem{MR3338680}
Ivan Losev and Ben Webster, \emph{On uniqueness of tensor products of
  irreducible categorifications}, Selecta Math. (N.S.) \textbf{21} (2015),
  no.~2, 345--377.

\bibitem{MR3765456}
Nathan Manning, Erhard Neher, and Hadi Salmasian, \emph{Integrable
  representations of root-graded {L}ie algebras}, J. Algebra \textbf{500}
  (2018), 253--302.

\bibitem{MR4681327}
Sudipta Mukherjee, Santosha~Kumar Pattanayak, and Sachin~S. Sharma, \emph{Weyl
  modules for toroidal {L}ie algebras}, Algebr. Represent. Theory \textbf{26}
  (2023), no.~6, 2605--2626.

\bibitem{MR360732}
A.~N. Rudakov, \emph{Irreducible representations of infinite-dimensional {L}ie
  algebras of {C}artan type}, Izv. Akad. Nauk SSSR Ser. Mat. \textbf{38}
  (1974), 835--866.

\bibitem{MR634510}
George~B. Seligman, \emph{Rational constructions of modules for simple {L}ie
  algebras}, Contemporary Mathematics, vol.~5, American Mathematical Society,
  Providence, RI, 1981.

\bibitem{MR607583}
I.~R. Shafarevich, \emph{On some infinite-dimensional groups. {II}}, Izv. Akad.
  Nauk SSSR Ser. Mat. \textbf{45} (1981), no.~1, 214--226, 240.

\bibitem{stacks-project}
The {Stacks project authors}, \emph{The stacks project},
  \url{https://stacks.math.columbia.edu}, 2024.

\bibitem{TanYangian}
Yilan Tan, \emph{Finite-dimensional representations of {Y}angians}, Ph.D.
  thesis, University of Alberta, 2014.

\bibitem{MR3334147}
Yilan Tan and Nicolas Guay, \emph{Local {W}eyl modules and cyclicity of tensor
  products for {Y}angians}, J. Algebra \textbf{432} (2015), 228--251.

\bibitem{wiggins_stratified}
Giulian Wiggins, \emph{Stratified categories and a geometric approach to
  representations of the {S}chur algebra}, Ph.D. thesis, The University of
  Sydney, 2022.

\end{thebibliography}
\providecommand{\bysame}{\leavevmode\hbox to3em{\hrulefill}\thinspace}
\providecommand{\href}[2]{#2}

\end{document}